\documentclass{amsart}
\usepackage{amsthm}
\usepackage{amsfonts}
\usepackage{amsmath}
\usepackage{caption}
\usepackage{graphicx}
\usepackage{comment}
\usepackage{amscd}
\usepackage{tikz-cd}

\newcommand{\GG}{\mathbb{G}}

\newcommand{\KK}{\mathbb{K}}
\newcommand{\QQ}{\mathbb{Q}}
\newcommand{\ZZ}{\mathbb{Z}}
\newcommand{\NN}{\mathbb{N}}

\newcommand{\VV}{\mathbb{V}}
\newcommand{\Aut}{\mathrm{Aut}}
\newcommand{\AAut}{\mathrm{AAut}}
\newcommand{\SAut}{\mathrm{SAut}}

\begin{document}

\author{Viktoriia Borovik}
\address{Lomonosov Moscow State University, Faculty of Mechanics and Mathematics, Department of Higher Algebra, Leninskie Gory 1, Moscow, 119991 Russia; \linebreak and \linebreak
HSE University, Faculty of Computer Science, Pokrovsky Boulevard 11, Moscow, 109028, Russia}
\email{vborovik@hse.ru}

\author{Sergey Gaifullin}
\address{Moscow Center for Fundamental and Applied Mathematics, Moscow, Russia; \linebreak and \linebreak
National Research University Higher School of Economics, Faculty of Computer Science, Pokrovsky Boulevard 11, Moscow, 109028, Russia}
\email{sgayf@yandex.ru}

\author{Anton Shafarevich}
\address{Moscow Center for Fundamental and Applied Mathematics, Moscow, Russia; \linebreak and \linebreak
National Research University Higher School of Economics, Faculty of Computer Science, Pokrovsky Boulevard 11, Moscow, 109028, Russia}
\email{shafarevich.a@gmail.com}

\subjclass[2020]{Primary 14R20,  14J50; Secondary 13A50, 14M25.}

\keywords{Horosherical variety, toric variety, automorphism, locally nilpotent derivation.}

\thanks{The second author was supported by RSF grant 20-71-00109.}

\title{On orbits of automorphism groups on horospherical varieties}
\maketitle 

\newtheorem{theorem}{Theorem}
\newtheorem{conj}{Conjecture}
\newtheorem{lemma}{Lemma}
\newtheorem{cor}{Corollary}
\newtheorem{proposition}{Proposition}
\theoremstyle{remark}
\newtheorem{remark}{Remark}
\theoremstyle{definition}
\newtheorem{definition}{Definition}
\newtheorem{ex}{Example}

\sloppy

\begin{abstract}
In this paper we describe orbits of automorphism group on a horospherical variety in terms of degrees of homogeneous with respect to natural grading locally nilpotent derivations. In case of (may be nonnormal) toric varieties a description of  orbits of automorphism group in terms of corresponding weight monoid is obtained.
\end{abstract}

\section{Introduction}

Let $\KK$ be an algebraically closed field of characteristic zero. If we are given by an affine algebraic variety $X$ we can consider the group of its regular automorphisms $\Aut(X)$.  This group naturally acts on $X$. We study orbits of this action.

If the variety $X$ admits an action of an algebraic group $G$, then $\Aut(X)$-orbits are unions of $G$-orbits. So, to describe $\Aut(X)$-orbits we are to obtain a criterium for two $G$-orbits to lie in the same $\Aut(X)$-orbit. This aproach is very usefull when there are only finite number of $G$-orbits. Often it is more convenient to describe orbits of the neutral component $\Aut(X)^0\subset\Aut(X)$. Arzhantsev and Bazhov~\cite{AB} described  $\Aut(X)^0$-orbits for normal toric varieties $X$, see also \cite{Sh2}. In this paper we obtain a generalization of this result. We investigate $\Aut(X)^0$-orbits on complexity-zero horospherical varieties $X$. Recall that {\it horospherical} variety is an irreducible variety admitting an action of an affine algebraic group such that the stabilizer of a generic point contains a maximal unipotent subgroup in $G$. It is called {\it complexity-zero} if $G$-action on $X$ has an open orbit.

The automorphism group of an affine variety usually is not an algebraic group. But we can consider the subgroup $\AAut(X)\subset \Aut(X)$ generated by all algebraic subgroups $H\subset \Aut(X)$. Every algebraic group is generated by subgroups isomorphic to additive and multiplicative group of $\KK$. We call such subgroups $\GG_a$ and $\GG_m$-subgroups respectively. So,  $\AAut(X)$ is the subgroup of $\Aut(X)$, generated by all $\GG_a$ and $\GG_m$-subgroups. Remark that the subgroup $\SAut(X)$ generated by all $\GG_a$-subgroups is called the subgroup of {\it special} automorphism. In \cite{AFKKZ} varieties $X$ with transitive action of $\SAut(X)$ on the smooth locus $X^{reg}$ were investigated.   Such varieties are called {\it flexible}. Arzhantsev, Kujumzhijan and Zaidenberg~\cite{AKZ} proved flexibility of normal toric varieties. Boldyrev and Gaifullin~\cite{BG} give a criterium for a (not nessesary normal) toric variety to be flexible. Shafarevich~\cite{Sh} proved flexibility of horospherical complexity-zero varieties corresponding to a semisimple group $G$. Gaifullin and Shafarevich~\cite{GSh} proved flexibility of normal horospherical complexity-zero varieties corresponding to an arbitrary group. So, if we have a normal horospherical complexity-zero variety, all regular points form one $\Aut(X)$-orbit. So in case of normal horospherical variety the goal is to describe singular $\Aut(X)$-orbits.

It is easy to see that $\AAut(X)$ is a subgroup of $\Aut(X)^0$. These groups can be different but we prove that their orbits for horospherical complexity-zero varieties coinside. Then we investigate $\AAut(X)$-orbits. To do this we need two opposite teqniques. We need to glue $G$-orbits if they lies in the same $\Aut(X)^0$-orbit and to separate $G$-orbits which are contained in different $\Aut(X)^0$-orbits. Some of $\GG_a$-orbits we can glue by $T$-normalized $\GG_a$-actions with respect to the right action of the maximal torus $T\subset G$. We prove that if two $G$-orbits can not be glued by any chain of $T$-normalized $\GG_a$-actions, then they lies in different  $\Aut(X)^0$-orbits.

We obtain a criterium for two $G$-orbits to be contained in one $\Aut(X)^0$-orbit in terms of degrees of $\mathfrak{X}(T)$-homogeneous locally nilpotent derivations, see Theorem~\ref{mmain}. But in arbitrary case the problem of describing these degrees is open. There is completely solved in case of (may be nonnormal) toric varieties. So, in this case we obtain a description of $\Aut(X)^0$-orbits in terms of the weight monoid corresponding to the variety, see Corollary~\ref{p}. 

\section{Preliminaries}
\subsection{Neutral component of automorphism group}

Let us define the connected component of $\Aut(X)$ following \cite{R}, see also \cite{AB}.

\begin{definition} A family $\{\varphi_b, b\in B\}$ of automorphisms of a variety $X$, where the parametrizing set $B$ is an algebraic variety, is an {\it algebraic family} if the map $B\times X\rightarrow X$ given by $(b,x)\mapsto \varphi_b(x)$ is a morphism.
\end{definition}

\begin{definition} The connected component $\Aut(X)^0$ of the group $\Aut(X)$ is the subgroup of
automorphisms that may be included in an algebraic family $\{\varphi_b, b\in B\}$ with an irreducible
variety as a base $B$ such that $\varphi_{b_0}=\mathrm{id}_X$ for some $b_0\in B$.
\end{definition}

It is easy to check that $\Aut(X)^0$ is indeed a subgroup, see~\cite{R}.

If $G$ is an algebraic group and $G\times X\rightarrow X$ is a regular action, then we may take $B = G$ and consider the algebraic family $\{\varphi_g, g\in G\}$, where $\varphi_g(x)=gx$. So any automorphism defined by an element of G is included in $\Aut(X)^0$. In partiqular every $\GG_a$ and every $\GG_m$-subgroup is contained in $\Aut(X)^0$. Therefore, $\AAut(X)\subseteq\Aut(X)^0$.

\subsection{Cones}\label{kd}

Let $M\cong \ZZ^n$ be a lattice and $P\subset M$ be a finite generated submonoid.  Let us consider the following  vector space over rational numbers  $M_{\QQ} = M \otimes_{\ZZ} \QQ$. The cone in $M_{\QQ} $ spanned by $P$ we denote by $\sigma^\vee=\sigma^\vee(P)$. It is a finitely generated polyhedral cone. The monoid $P$ is called {\it saturated}, if $P=\mathbb{Z}P\cap \sigma^{\vee}$. 

Let $N=\mathrm{Hom}(M,\ZZ)$ be the dual lattice. Denote by $\langle \cdot, \cdot \rangle : M \times N \rightarrow \ZZ$ the natural pairing between these lattices. It extends to the pairing $\langle \cdot, \cdot \rangle_{\QQ} : M_{\QQ} \times N_{\QQ} \rightarrow \QQ$ between the vector spaces $M_{\QQ}$ and $N_{\QQ} = N \otimes_{\ZZ} \QQ$. Let us denote by $\sigma$ the cone dual to $\sigma^\vee$
$$
\sigma=\{u\in N_\QQ\mid \langle v,u\rangle\geq 0\text{ for all } v\in \sigma^\vee\}.
$$
There is a natural bijection between $k$-dimensional faces of $\sigma$ and $(n-k)$-dimensional faces of $\sigma^\vee$. A face $\tau\preceq\sigma$ corresponds to the face $\widehat{\tau}=\sigma^\vee\cap\langle \tau\rangle^\bot$.

A cone is called {\it pointed} if $\sigma\cap(-\sigma)=\{0\}$. The cone $\sigma$ is pointed if and only if the cone $\sigma^\vee$ is of full dimension, i.e. the linear shell $\langle \sigma^\vee\rangle=M_\QQ$. If $\sigma^\vee$ is not of full dimension, we replace $M$ by the group generated by $P$. So, further we assume that $\sigma$ is pointed.

Let us denote the set of rays of the cone $\sigma$ by  $\sigma(1)$ and let $p_{\rho}$ be the primitive lattice vector on a ray $\rho$. 
\begin{definition}
Let $$\mathfrak{R}_{\rho}:=\{e \in M \ \ | \ \ \langle e,p_{\rho} \rangle = -1, \langle e,p_{\rho '} \rangle \geq 0 \ \ \forall \rho ' \neq \rho \in \sigma(1)\}.$$
Then the elements of the set $\mathfrak{R} := \bigsqcup\limits_{\rho}\mathfrak{R}_{\rho}$ are called the {\it Demazure roots} of the cone~$\sigma$.
\end{definition}
We will call ray $\rho$  the {\it distinguished ray} of the Demazure root $e$ if $e \in \mathfrak{R}_{\rho}$.

\begin{definition}
Let $\tau$ be a face of $\sigma^\vee$. Let $\rho_1,\ldots, \rho_k$ be all rays normal to $\tau$ and $p_1,\ldots, p_k$ be their primitive vectors. Suppose $e$ is a Demazure root such that $\langle e, p_{1}\rangle=-1$ and  $\langle e, p_{i}\rangle=0$ for all $2\leq i\leq k$. Then we say that $e$ is a $\tau$-{\it root}. 
\end{definition}

Let us give some definition according to \cite{TY}.

\begin{definition}
An element $p$ of the monoid $P$ is called {\it saturation point} of $P$, if the moved cone $p+\sigma^{\vee}$ has no holes, i.e. $(p+\sigma^{\vee})\cap M\subset P$. 
 
A face $\tau$ of the cone $\sigma^{\vee}$ is called {\it almost  saturated}, if there is a saturation point of  $P$ in $\tau$. 
Otherwise $\tau$ is called a {\it nowhere saturated} face.
\end{definition}

The following lemma is known, see for example~\cite[Lemma~2]{BG}.
\begin{lemma}\label{mg}
The maximal face, i.e. the whole cone $\sigma^\vee$, is almost saturated.
\end{lemma}

\subsection{Horospherical varieties}\label{HV}

We recall some results on horospherical varieties, all proofs can be found in \cite{PV}, see also \cite{T}.

Let $G$ be a connected linear algebraic group. 

\begin{definition}
An irreducible $G$-variety $X$ is called \emph{horospherical}, if for a generic point $x\in X$ the stabilizer of $x$ contains a maximal unipotent subgroup $U\subseteq G$.

If $X$ contains an open $G$-orbit, then $X$ is called \emph{complexity-zero horospherical}. In \cite{PV} affine complexity-zero horospherical varieties are called $S$-varieties.
\end{definition}

Suppose that $X$ is an affine complexity-zero horospherical variety. It is easy to see that the unipotent radical of $G$ acts trivially on $X$. Hence we may assume that $G$ is reductive. Taking a finite covering, we may assume that $G=T\times G'$, where $T$ is an algebraic torus and $G'$ is a semisimple group.

Let $O$ be the open orbit in $X$. We have the following sequence of inclusions
$$
\KK[X]\hookrightarrow\KK[O]\hookrightarrow\KK[G].
$$

Let $B$ be a Borel subgroup of $G$ and let $M=\mathfrak{X}(B)$ be the group of characters of $B$.   
For a $\Lambda\in M$ we put
$$
S_\Lambda=\left\{ f\in\KK[G]\mid f(gb)=\Lambda(b)f(g)\ \text{for all} \ g\in G, b\in B\right\}.
$$
Then
$$
S_{\Lambda} S_{\Lambda'}=S_{\Lambda+\Lambda'}.
$$

The set $\mathfrak{X}^+(B)$ of dominant weights consists of all $\Lambda$ such that $S_\Lambda\neq\{0\}$. 
It is proved in \cite{PV} that for an affine complexity-zero horospherical $G$-variety $X$ there is a decomposition
$$
\KK[X]=\bigoplus_{\Lambda\in P} S_\Lambda
$$
for some submonoid $P\in\mathfrak{X}^+(B)$. 

Using notations from the previous section, we denote by $\sigma^\vee$ the cone in $M_\QQ$ spanned by $P$. The variety $X$ is normal if and only if $P$ is saturated. There is a one-to-one correspondence between faces of $\sigma$ and $G$-orbits of $X$. More precisely, if
$O_\tau\subseteq  X$ is the $G$-orbit in $X$ corresponding to a face $\tau$ of the cone $\sigma$, then the ideal of functions vanishing on $O_\tau$ has the form 
$$
I(O_\tau)=\bigoplus_{\Lambda\in P\setminus\tau}S_\Lambda.
$$ 
This ideal vanishes on the closure $\overline{O_\tau}$. Then 
$$O_\tau=\overline{O_\tau}\setminus  \left(\bigcup_{\gamma \prec \tau}\overline{O_\gamma} \right).$$
If $\widehat{\xi} \preceq \sigma^\vee$ is the face corresponding to $\xi \preceq \sigma$, we use both denotations: $O_\xi=O_{\widehat{\xi}}$.

\begin{remark}
The cone $\sigma^\vee$ can be not of full dimension. But if we replace $M$ by the group generated by $P$, the correspondence between faces and orbits remains.
\end{remark}

To obtain a variety $X$ explicitly one should considet generators $\Lambda_1,\ldots, \Lambda_m$ of $P$ and consider the sum of irreducible $G$-representation which are contragradient to ones with highest weights  $\Lambda_1,\ldots, \Lambda_m$. In each $V(\Lambda_i)^*$ one need to find the eigenvector $v_i$. Put $v=v_1+\ldots+v_m$. Then $X\cong \overline{Gv}$. If $m=1$, then the variety $X$ is the closure of the eigenvector of an irreducible representation. Such varieties are called $HV$-varieties.  

An important partiqular case of horospherical varieties give toric varieties. More information on toric varieties one can find in \cite{CLSch} and \cite{Ful}.

\begin{definition}
A {\it toric variety} is a variety $X$ admitting an action of an algebraic torus $T \simeq (\KK^{\times})^n$ with open orbit.
\end{definition}
\begin{remark} Often by toric variety one mean a normal toric variety. We do not a-priori assume a toric variety to be normal.
\end{remark}

So, toric variety is a horospherical complexity-zero variety corresponding to $G\cong (\KK^{\times})^n$.
For a toric variety each nonzero homogeneous component has dimension one. We have
$$A = \bigoplus_{m \in P}\KK \chi^{m},$$
where  $\chi^m = t_1^{m_1}\cdot ... \cdot t_n^{m_n}$ is the character of the torus $T$ corresponding to a point $m = (m_1, ..., m_n)$.

\subsection{Derivations}

We recall basic facts from theory of locally nilpotent derivations, see for example \cite{Fr}. 

Let $A$ be a commutative associative algebra over $\KK$.
\begin{definition}
A linear mapping $\partial : A \rightarrow A$ is called {\it a derivation} if it satisfies the Leibniz rule: $\partial(ab) = a\partial(b) + b\partial(a)$.

A derivation is called {\it locally nilpotent or LND} if for any $a \in A$ there is $n \in \NN$ such that $\partial^n(a) = 0$.

A derivation is called {\it semisimple} if there exists a basis of $A$ consisting of $\partial$-semi-invariants (i.e. $ \partial(a) = \lambda a$ for $a \in A$).
\end{definition}

\begin{definition}
A derivation $\partial : A \rightarrow A$ is called {\it locally bounded} if any element $a \in A$ is contained in a $\partial$-invariant finite-dimensional linear subspace $V \subset A$.
\end{definition}
\begin{remark}
One can see that all semisimple and locally nilpotent derivations are locally bounded.
\end{remark}

Exponential mapping gives a correspondence between LND and $\GG_a$-subgroups in $\Aut(A)$, and between semisimple derivations and $\GG_m$-subgroups in $\Aut(A)$. A derivation $\delta$ correspondes to the subgroup $\{\exp(t\delta)\}$, where $t\in \KK$ for $\GG_a$-subgroup and $t\in \KK^\times$ for $\GG_m$-subgroup.

Let $F$ be an abelian group. Consider $F$-grading:

$$A = \bigoplus_{f \in F}A_f, \ \ \ A_fA_g \subseteq A_{f+g}.$$

\begin{definition}
A derivation $\partial : A \rightarrow A$ is called {\it $F$-homogeneous of degree $f_0 \in F$} if for all $a \in A_f$ we have $\partial(a) \in A_{f+f_0}.$
\end{definition}

Now let $A$ be a finitely generated $\ZZ$-graded algebra. The following lemmas are known.
\begin{lemma} \label{l1}
Let $\partial$ be a derivation of A. Then $\partial = \sum\limits_{i=l}^k \partial_i$, where $\partial_i$ is the homogeneous derivation of degree $i$.
\end{lemma}

\begin{lemma} \label{l2}
Let $\partial = \sum\limits_{i=l}^k \partial_i$ be a derivation of A. Then:
\begin{enumerate}
\item If $\partial$ is LND then $\partial_l$ and $\partial_k$ are LNDs.
\item If $\partial$ is locally bounded then if $l \neq 0$, $\partial_l$ is LND, and if  $k \neq 0$, $\partial_k$ is LND.
\end{enumerate}
\end{lemma}

\begin{cor}
If $A$ admits an LND, then $A$ admits a $\ZZ$-homogeneous LND.
\end{cor}

Now let $A$ be a $\ZZ^n$-graded algebra. If $\delta$ is a derivation it can be decomposed onto a sum of $\ZZ^n$-homogeneous derivations. Let $R$ be the set of $\ZZ^n$-degrees of nonzero summands.  
If $\partial$ is locally bounded, then summands corresponding to vertices of the convex hull different from the origin are LNDs. This implies the following lemma.

\begin{lemma}\label{lleemm}
Let $\partial$ be a $\ZZ$-homogeneous of degree $d\neq 0$ locally bounded derivation. Suppose $\ZZ$ is included as a subgroup to $\ZZ^n$. Then among $\ZZ^n$-homogeneous summands of $\partial$ there is an LND $\delta\neq 0$. And $\ZZ$-degree of $\delta$ equals to $d$.
\end{lemma} 

Now let $X$ be an affine toric variety. Fix a Demazure root $e \in \mathfrak{R}_{\rho}$.
One can define the following $M$-homogeneous LND $\partial_e$ on the algebra $A=\KK[X]$ by the rule
$$\partial_e(\chi^m) = \langle p_{\rho}, m \rangle \chi^{e+m}.$$
Any homogeneous LND on $A = \KK[X]$ has the form $\lambda \partial_e$ for some $\lambda \in \KK, e \in \mathfrak{R}$.

\section{Varieties with finite numbers of $G$-orbit}\label{FNO}
Let $X$ be an irreducible affine variety. Suppose a connected linear algebraic group $G$ acts on $X$ with finite number of $G$-orbits. Then the image of $G$ in $\Aut(X)$ is contained in $\AAut(X)\subset\Aut(X)^0$. Therefore, there is only finite number $\AAut(X)$-orbits on $X$. It is easy to see that $\AAut(X)$ is a normal subgroup of $Aut(X)$. Therefore, each automorphism permutes $\AAut(X)$-orbits. Hence, the connected group $\Aut(X)^0$ preserves each $\AAut(X)$-orbit. That is $\Aut(X)^0$-orbits coincides with $\AAut(X)$-orbits. 

Consider a $G$-orbit $Z$. Denote 
$$\Omega = \{ O\text{ is a } G-\text{orbit} |  Z \subseteq O\ \text{and the closure}\  \overline{O}\ \text{is}\ \AAut (X)\text{-invariant}\}.$$ 

\begin{lemma}\label{antonlemma}

There is a $G$-orbit $\Phi(Z)\in\Omega$ such that 
$$\overline{\Phi(Z)} = \bigcap_{O \in \Omega} \overline{O}.$$ 

\end{lemma}

\begin{proof}
Indeed, the set $ \bigcap\limits_{O \in \Omega} \overline{O}$ is closed $G$-invariant set so it is a finite union of closures of $G$-orbits. 

$$ \bigcap_{O \in \Omega} \overline{O}= \overline{O_{1}}\cup \ldots \cup \overline{O_{k}}.$$

In the same time $\AAut(X)$ is a connected group. So each $\overline{O_{i}}$ is $\AAut(X)$-invariant. But $Z$ is irreducible. So there is a number $j$ such that $Z \subseteq \overline{O_{j}}$. Then $O_j\in \Omega$ and $\Phi(Z)=O_j$.

\end{proof}

\begin{proposition}
Let $Y$ be $\AAut(X)$-orbit such that $Z\subseteq Y$. Then $Y$ contains~$\Phi(Z)$.
\end{proposition}

\begin{proof}
Otherwise $\AAut(X)$-orbit containing $Z$ is contained in the set $\overline{\Phi(Z)}\setminus \Phi(Z)$. Then~$Y$ is a union of $G$-orbits $O_{1}\cup\ldots \cup O_{k}$ where $O_{j} \subseteq \overline{\Phi(Z)}$ for all~$j$. Since the group $\AAut (X)$ is connected, each $O_{j}$ is $\AAut(X)$-invariant. This implies that $\overline{O_{j}}$ is $\AAut(X)$-invarint for each $j$. But $Z$ is irreducible so there is a $j$ such that $Z \subseteq \overline{O_{j}}$. But this contradicts to the definition of $\Phi(Z)$. 
\end{proof}

Summerizing the preveous results we obtain the following corollary.
\begin{cor}
Let $Y$ be $\Aut(X)^0$-orbit containing $Z$. Denote 
$$S=\{O\subset \overline{\Phi(Z)}\mid Z\text{ is not contained in } \overline{\Phi(O)}\}.$$ 
Then $Y=\overline{{\Phi(Z)}}\setminus \bigcup\limits_{O\in S} \overline{O}$.
\end{cor}

Note that $S=\{O\subset \overline{\Phi(Z)}\mid \Phi(O)\neq \Phi(Z)\}$. Therefore, we obtain the folowing assertion.
\begin{cor}
Let $O_1$ and $O_2$ be $G$-orbits. Then $\Aut^0 (X)$-orbits containing $O_{1}$ and $O_{2}$ coincide if and only if $\Phi(O_1) = \Phi(O_2)$.
\end{cor}

We assume that we know adjunction of $G$-orbits.  Results of this section show that to obtain description of $\Aut^0(X)$-orbits it is sufficient to determine which $G$-orbits have $\AAut(X)$-invariant closures.

\section{Automorphisms of a variety with torus action}\label{TAC}

Let $X$ be an affine algebraic variety admitting an effective $T\cong(\KK^\times)^n$-action. The group of characters $M=\mathfrak{X}(T)$ is isomorphic to $\ZZ^n$.   The $T$-action corresponds to an $M$-grading on $A=\KK[X]$. Let $P$ be the weight monoid of this action i.e. $P=\left\{m\in M\mid A_m\neq\{0\}\right\}$. We use notations from Section~\ref{kd}.

\begin{definition}
Let $\tau$ be a face of $\sigma^\vee$.  
An element $v\in N$ is called {\it  $\tau$-bordering} if the following conditions are satisfied  
\begin{itemize}
\item $\langle \omega,v\rangle>0$ for all $\omega\in P\setminus \tau$.
\item if $\partial$ is a nonzero $M$-homogeneous LND of $A$ of degree $e$, then $\langle e,v\rangle\geq 0$.
\end{itemize}
\end{definition}
Let us consider the ideal $I_\tau=\bigoplus\limits_{\omega\in P\setminus\tau}A_\omega$ of $A$. Denote $Z_\tau=\VV(I_\tau).$
\begin{proposition}\label{be}
 If a face $\tau$ of $\sigma^\vee$ admits a $\tau$-bordering element $v\in N$, then the set $Z_\tau$ is $\AAut(X)$-invariant.
\end{proposition}
\begin{proof}
Let us consider the following $\ZZ$-grading on $A$:
$$
A=\bigoplus_{i\in\ZZ}A_{i},\qquad\text{where }A_{i}=\bigoplus_{\langle \omega, v\rangle=i}A_\omega.
$$

Let $\partial$ be a locally bounded derivation. By Lemma~\ref{l1} we have $\partial = \sum\limits_{i=l}^k \partial_i$, where $\partial_i$ is homogeneous under $\ZZ$-grading of degree $i$. If $l < 0$ by Lemma~\ref{l2} we obtain that $\partial_l$ is LND. Then we can decompose $\partial_l$ into the sum of $M$-homogeneous LNDs:
$$\partial_l = \sum\limits_{j}\partial_{lj},$$
where $\ZZ$-degree of $\partial_{lj}$ is equal to $l$. By Lemma~\ref{lleemm}  there exists an $M$-homogeneous LND $\partial_{lj}$. We obtain a contradiction. Therefore, $l  \geq 0$.

If $f\in I$, then it can be decomposed into the sum of $\ZZ$-homogeneous elements of positive degree $f=f_1+\ldots+f_r$. Therefore, 
$$\partial(f)=\sum_{i>0}\sum_{j\geq l\geq 0}\partial_j(f_i)\in  \bigoplus\limits_{p >0}A_p=I.$$
So, the ideal $I$ is $\partial$-invariant for any semisimple and for any locally nilpotent derivation. This implies that $Z_{\tau}$ is $\AAut(X)$-invariant.
\end{proof}

Let $\rho_1, \ldots, \rho_k$ be the rays of $\sigma$. And let $p_i$ be the primitive vectors on $\rho_i$.
Now we suppose that the $M$-grading has the following property:
\begin{equation}\label{oo}
A_\alpha\cdot A_\beta=A_{\alpha+\beta}\qquad\text{for all }\alpha,\beta\in P.
\end{equation}
The former equality means $A_{\alpha+\beta}=\langle fg\mid f\in A_\alpha, g\in A_\beta\rangle$.

\begin{lemma}\label{dr}
Let $e\in M$ is the degree of a $M$-homogeneous LND $\partial$ such that $e\notin \sigma^\vee$. Then $e$ is a Demazure root of $\sigma$. 
\end{lemma} 
\begin{proof}
If $\omega\in P$ is such an element that $\omega+e\notin P$. Then $\partial(A_\omega)=0$. 

Suppose $\langle e,p_i\rangle=-d\leq -2$.  Then there is $\widehat{\omega}\in P$ which is an inner point of $\sigma^\vee$ such that $A_{\widehat{\omega}}\in \mathrm{Ker}\,\partial$. Indeed, let us take a saturation point $u$ of $\sigma^\vee$. Let us take an element $v\in u+\sigma^\vee\cap M\subset P$ such that $v$ is an inner point of $\sigma^\vee$ and $\langle v, p_i\rangle$ is not divisible by $d$. Consider the sequence of points $v, v+e, v+2e,\ldots$ This  sequence leave $\sigma^\vee$. Hence there exists minimal $k$ such that $v+ke\notin P$. Then $v+(k-1)e$ is an inner point of $\sigma^\vee$ such that $A_{v+(k-1)e}\in\mathrm{Ker}\,\partial$.

Then there is $l\in \mathbb{N}$ such that $l\widehat{\omega}-u-\beta \in   \sigma^\vee$. 
Hence $l\widehat{\omega}-\beta \in u+ \sigma^\vee\cap P$.
Therefore, there is $\alpha\in P$ such that  $A_\beta A_\alpha=A_{\alpha+\beta}=A_{l\widehat{\omega}}=A^l_{\widehat{\omega}}\subset\mathrm{Ker}\,\partial$. By \cite{Fr} the kernel of any LND is factorially closed. Therefore $A_{\beta}\subset \mathrm{Ker}\,\partial$. So, $\partial=0$.

Suppose $\langle e,p_i\rangle\leq -1$, $\langle e,p_j\rangle\leq -1$. Then each $\alpha\in \langle p_i\rangle^\bot\cap P$ and each $\beta\in  \langle p_j\rangle^\bot\cap P$ are in the kernel of $\partial$. There exist such $\alpha$ and $\beta$ that $\widehat{\omega}=\alpha+\beta\in\mathrm{Ker}\,\partial$ is an inner point of~$\sigma^\vee$. Then we again obtain a contradiction.

Thus, there exists unique $i$ such that $\langle e,p_i\rangle\leq -1$ and for all other $j$ we have $\langle e,p_j\rangle\geq 0$.
\end{proof}

\begin{proposition}\label{prep}
Let $\tau$ be a face of $\sigma^\vee$. The subset $Z_\tau$ is not $\AAut(X)$-invariat if and only if there is a nonzero M-homogeneous LND with degree $e$, such that $e$ is a $\tau$-root. \end{proposition}
\begin{proof}
Suppose $Z_\tau$ is not $\AAut(X)$-invariant. Let us consider $p=\sum\limits_{p_k\bot \tau} p_k$. By Lemma~\ref{be}  $p$ is not a $\tau$-bordering element. Hence, there exists a nonzero $M$-homogeneous LND of degree~$e$ such that $\langle e,p\rangle<0$. By Lemma~\ref{dr} the element $e$ is a Demazure root. Therefore, there is $i$ such that $\langle e,p_i\rangle=-1$ and
$$
0>\langle e,p\rangle=\sum_{p_k\bot \tau} \langle e,p_k\rangle=-1+\sum_{j\neq i, p_j\bot \tau} \langle e,p_j\rangle,
$$
where $ \langle e,p_j\rangle\geq 0$ for all $j\neq i$. Hence, $\langle e,p_j\rangle=0$ for all $j\neq i$.

Now suppose there exists a nonzero M-homogeneous LND $\partial$ with degree $e$, where $e$ is a $\tau$-root. Then there is $i$ such that $p_i\bot \tau$ and $\langle e,p_i\rangle=-1$. And for all $i\neq j$ such that $p_j\bot \tau$ we have $\langle e,p_j\rangle=0$. Denote  $\gamma=\langle p_i\rangle^\bot\cap\sigma^\vee$. Then for every $u\in \gamma$ we have $\partial(A_u)=\{0\}$. Let $q_1,\ldots, q_s$ be primitive vectors on all rays of $\sigma$ which are not normal to $\tau$. Denote $c_r=\langle e, q_r\rangle$, $1\leq r\leq s$. There exists $\omega\in P\cap \tau$ such that $\langle \omega, q_r\rangle>c_r$ for all $r$. Therefore, $\omega-e\in P$. Note that $A_\omega\in I_\tau$. If there is $f\in A_\omega$ such that $\partial(f)\neq 0$, then $\partial(f)\notin I_\tau$. Hence, $I_\tau$ is not $\partial$-invariant. That is $Z_\tau$ is not $\AAut(X)$-invariant. Assume $A_\omega\in\mathrm{Ker}\,\partial$. Then there exists $u\in \gamma$ such that $u+\omega$ is an interior element of $\sigma^\vee$. As in Lemma~\ref{dr} we obtain $\partial=0$.
\end{proof}

\begin{remark}\label{vz}
In situation of the proof of  Proposition~\ref{prep}, denote $\xi_i=\sigma^\vee\cap\langle p_j|j\neq i\rangle^\bot$. In the proof of Proposition~\ref{prep} we see that if $Z_\tau$ is not $\AAut$-invariant, then for some $i$ there is $f\in A_\omega$, $\omega\in\xi_i$ such that $\partial(f)\notin I_\tau$. It is easy to see that we can assume $\omega$ to be an interior vector of $\xi_i$. Therefore, if $\partial(f)\in I_\zeta$ for some face $\zeta\preceq \sigma^\vee$, then $\xi_i$ is contained in $\zeta$.  Hence, for some $x\in \overline{O_{\tau}}$ and for some $t\in\KK$ the point $\exp(t\partial)(x)\in O_\zeta$ for some face $\zeta$ containing $\xi_i$.
\end{remark}

\section{Automorphism orbit on horospherical varieties}

Let $X$ be a horospherical complexity-zero variety. In Section~\ref{HV} we introduce a grading on $A=\KK[X]$ and describe $G$-orbits on $X$. So, we can apply to the variety $X$ results of Sections ~\ref{TAC} and~\ref{FNO}. Combining these results we obtain the folowing theorem.

\begin{theorem}\label{mmain}
Let $X$ be a horospherical complexity-zero variety. Each closure of $\Aut(X)^0$-orbit has the form
$\overline{O_\tau}$, where $\tau$ is a face of $\sigma^\vee$ such that there is no any nonzero $M$-homogeneous LND with degree equals a $\tau$-root.
\end{theorem}

The following corollary follows from this theorem and Remark \ref{vz}.
\begin{cor}
Let $H$ be the subgroup generated by $G$ and all exponents of $M$-homogeneous LNDs. Then $\Aut(X)^0$-orbits on $X$ coincides with $H$-orbits.
\end{cor}

If we are given by a face $\tau$ of $\sigma^\vee$, it is an easy question if $\tau$ admits a $\tau$-root. But the problem is that for a given $\tau$-root $e$ not always there exists an $M$-hompgeneous LND with degree $e$. 

\begin{ex} 
Let $G=\mathrm{SL}_3$ and $P=\mathfrak{X}^+(B)$. Then the cone $\sigma^\vee$ in basis of fundamental weights is $\mathrm{cone}(e_1,e_2)$. Each face admits an $\tau$-root. For $\tau$ equals the origin, there are two $\tau$-roots: $(-1,0)$ and $(0,-1)$.  It is easy to compute that $X\cong\VV(x_1y_1+x_2y_2+x_3y_3)\subset \KK^6$. The orbit, corrasponding to $\tau$ is the point $q=(0,0,0,0,0,0)$. It is easy to see, that there are no LND with degrees $\tau$-roots. And the point $q$ is $\Aut(X)$-stable since it is the unique singular point. 
\end{ex}

\begin{ex}
We have the similar situation in case of HV-varieties. In this case the cone $\sigma^\vee$ is a ray. Therefore, for $\tau$ equals the origin there is a $\tau$-root. But for all HV-varieties exept affine space, $\tau$ correspondes to the unique singular point.
\end{ex}

The question for which $\tau$ there exists a $M$-homogeneous LND with the degree equals to a $\tau$-root Let us formulate a conjecture about answer to this question. 

\begin{conj}
Let $\tau$ be a face of $\sigma^\vee$. Let $p_1,\ldots, p_k$ be all primitive vectors on rays $\rho_1, \ldots, \rho_k$ of $\sigma$ that are normal to $\tau$. For $1\leq i\leq k$ we denote 
$$\xi_i=\sigma^\vee\cap\langle p_1,\ldots, p_{i-1},p_{i+1},\ldots, p_k\rangle^\bot.$$ 
The closure of $G$-orbit $O_\tau$ is not $\AAut(X)$-invariant if and only if the following conditions occur

\begin{enumerate}
\item  there exists $1\leq a\leq k$ and a $\tau$-root with distinguished ray $\rho_a$.
\item points of $O_\tau$ and $O_\xi$ have equal dimensions of tangent spaces $T_xX$.
\end{enumerate}
\end{conj}
\begin{proof}[Proof of necessity]
By Proposition~\ref{prep} if the set $Z_\tau=\overline{O_\tau}$ is not $\AAut(X)$-invariant, then there is a nontrivial $M$-homogeneous LND $\partial$ with the degree $e$, where $e$ is a $\tau$-root. 
Then the condition (1) occurs. Remark~\ref{vz} implies that for some $x\in \overline{O_{\tau}}$ and for some $t\in\KK$ the point $\exp(t\partial)(x)\in O_\zeta$ for some face $\zeta$ containing $\xi_a$. Since $O_\tau$ is open in $ \overline{O_{\tau}}$, we can assume $x\in O_\tau$. Therefore, points of $O_\tau$ and $O_\zeta$ have equal dimensions of tangent spaces. Since $O_\tau\subset \overline{O_{\xi_a}}\subset  \overline{O_{\zeta}}$, if $x\in O_{\tau}$, $y\in O_{\xi_a}$, $z\in O_\zeta$, then $\mathrm{dim}\,T_xX\geq\mathrm{dim}\,T_yX\geq\mathrm{dim}\,T_zX$. But $\mathrm{dim}\,T_xX=\mathrm{dim}\,T_zX$. Therefore, $\mathrm{dim}\,T_xX=\mathrm{dim}\,T_yX$, i.e. condition (2) holds.

\end{proof}

But for subclass of toric varieties the answer is known. If $X$ is a normal toric variety, then each Demazure root is the degree of a unique up to multiplicative constant nonzero LND. Therefore we obtain the following corollary of Theorem~\ref{mmain}.

\begin{cor}
Let $X$ be a normal toric variety. Each closure of $\Aut(X)^0$-orbit has the form
$\overline{O_\tau}$, where $\tau$ is a face of $\sigma^\vee$ such that there is no any $\tau$-roots.
\end{cor}

This result follows from results of \cite{AB} and \cite{Sh2}.

Now let  $X$ be a nonnormal toric variety. Let $e$ be a Demazure root of $\sigma^\vee$. It is proved in \cite{BG}[Section~4] that there exists a nonzero $M$-homogeneous LND with degree $e$ if and only if $(e+P) \cap \sigma^\vee\subset P$. Let us call such Demazure roots {\it admissible}. Then we obtain the following corollary of Theorem~\ref{mmain}.

\begin{cor}\label{p}
Let $X$ be a (may be nonnormal) toric variety. Each closure of $\Aut(X)^0$-orbit has the form
$\overline{O_\tau}$, where $\tau$ is a face of $\sigma^\vee$ such that there is no any admissible $\tau$-roots.
\end{cor}

\end{document}